\documentclass[12pt]{article}
\usepackage{amsmath,amsthm,amssymb,color,amscd}

\def\vv{{\underline{v}}}

\def\nunu{\underline{\nu}}
\def\tt{{\underline{t}}}

\newcommand{\Var}{\ensuremath{\mathcal{V}_{\mathbb{C}}}}
\def\1{\underline{1}}
\def\0{\underline{0}}
\def\R{\mathbb R}
\def\P{\mathbb P}

\def\L{\mathbb L}

\def\LLL{\mathbb L}
\def\Z{\mathbb Z}

\def\C{\mathbb C}
\def\K{\mathbb K}

\def\J{\mathbb J}
\def\AA{{\mathcal A}}
\def\OO{{\mathcal O}}
\def\EE{{\mathcal E}}

\def\CC{{\mathcal C}}

\def\calE{{\mathcal E}}

\newtheorem{theorem}{Theorem}

\newtheorem{lemma}{Lemma}

\theoremstyle{definition}
\newtheorem{definition}{Definition}
\newtheorem{remark}{Remark}

\newtheorem{example}{Example}

\title{On a generalized Poincar\'e series of plane valuations
}

\author{
F.~Delgado, \thanks{The
first author was partially supported by the
grant PID2022-138906NB-C21 funded by
MICIU/AEI/10.13039/501100011033 and by ERDF/EU.}
\and S.M.~Gusein-Zade\thanks{
The work of the second author is an output of a research project implemented as part of the Basic Research Program at the National Research University Higher School of Economics (HSE University).
} }

\date{}

\begin{document}

\def\eps{\varepsilon}

\maketitle

\begin{abstract}
Earlier, there were defined two generalized (``motivic'') versions of the
Poincar\'e series of a collection of plane valuations on the algebra
$\OO_{\C^2,0}$ of germs of holomorphic functions in two variables.
One of them was defined
as an integral with respect to the generalized Euler characteristic over the
projectivization
of the extended semigroup of the collection. One has a natural version of it
for valuations on the algebra $\EE_{\K^2,0}$ of germs of holomorphic functions
in two variables whose Taylor coefficients are from a fixed subfield $\K$ of the
field $\C$ of complex numbers.
In this setting the usual Poincar\'e series were computed for one plane curve or
divisorial valuation on $\EE_{\K^2,0}$.
We give equations for the corresponding generalized Poincar\'e series.
\end{abstract}

%%%%%%%%%%%%%%%%%%%%%%%%%%%%%%%%
\section{Introduction}\label{sec:intro} %% 1
%%%%%%%%%%%%%%%%%%%%%%%%%%%%%%%%
In some cases, the adjective ``generalized'' with respect to some invariants is used in the following sense.
These invariants are (or can be) defined in terms of the Euler characteristic.
The Euler characteristic (defined as the alternating sum of the ranks of the cohomology groups with compact support)
is an additive invariant of topological spaces.
For some classes of spaces there are other (more fine) additive invariants.
For example, the Hodge--Deligne polynomial is an additive invariant
on the class of complex quasiprojective varieties.
The most fine (universal) additive invariant on this class of spaces
is the class $[X]$ of a variety $X$ in the Grothendieck ring
$K_0(\Var)$ of complex quasiprojective varieties.
The class $[X]$ is sometimes called {\em the generalized Euler characteristic} of $X$ (and sometimes, say, in the context of motivic integration, denoted by $\chi_g(X)$).
If, in a definition of an invariant, the usual Euler characteristic $\chi(\bullet)$ is, in a way, exchanged by the generalized Euler characteristic $\chi_g(\bullet)$,
the obtained invariant may be called a generalized (or motivic) one.
As examples, one can indicate the motivic monodromy zeta function \cite{DLzeta} (in this case the transition from the usual Euler characteristic to the generalized one is not very straightforward) or the generalized orbifold Euler characteristic \cite{GLM}.

An invariant which can be expressed in terms of the Euler characteristic is the Poincar\'e series of a collection of valuations on the algebra $\OO_{\C^2,0}$
of germs of holomorphic functions on the complex plane.
It was defined (in other terms) in~\cite{CDK} and
reformulated in terms of (different) integrals with respect to the Euler characteristic in~\cite{CDG-RMS-2000}
and~\cite{Duke}. Two versions of the
generalized Poincar\'e series corresponding to these two formulations were defined in~\cite{Monats}.
One of them ($P_g(\tt)$, $\tt=(t_1, \ldots, t_r)$, $r$ is
the number of valuations in the collection) was defined
as an integral with respect to the generalized Euler characteristic
(properly defined) 
over the (infinite-dimensional) projectivization
of the space of germs of functions.
The other one ($\widehat{P}_g(\tt)$)
was defined as an integral with respect to the generalized Euler characteristic
over the projectivization of the so called extended semigroup of the collection
of valuations. The coefficients of the second one
are polynomials in the class $\LLL$ of the complex
affine line (and therefore it is reasonable to devote it by
$\widehat{P}(\tt; \LLL)$); the coefficients of the first one are
polynomials in the inverse $q=\LLL^{-1}$ of the class $\LLL$ (therefore it is reasonable to devote it by
$P(\tt; q)$). (For some convenience of further considerations we omit here
the lower index $g$ at $P$ (for ``generalized'') used in~\cite{Monats}. The presence of the arguments $\LLL$ and $q$ already indicates that one consideres generalized versions
of the Poincar\'e series.) Both of these series appeared to be rather complicated
for computations.

The paper~\cite{Monats}
was mainly devoted to a computation
of the generalized Poincar\'e series
$P(\tt; q)$ for a collection of curve or
divisorial plane valuations. There were given complicated
(one can say~--- not really computable) equations for it.
This series appeared to be related with the so called St\"ohr zeta function~\cite{Stohr}
and the notion can be extended to arbitrary fields, say, to finite ones:~\cite{Moyano-Zuniga}.
Further computations (reformulations) of the equation for this series were made
in~\cite{GN}. It was shown that, for a collection of plane curve valuations,
this series is related with the generating series of the ranks of the Heegaard-Floer homologies of the corresponding algebraic link.
We shall not consider this version of the generalized Poincar\'e series in this paper.

In~\cite{Monats}, the series $\widehat{P}(\tt; \LLL)$ was computed only for one (somewhat artificial) situation.
Let $\pi:(W,D)\to(\C^2,0)$ be a modification
of the plane $(\C^2,0)$ by a sequence of blow-ups. Let $E_{\sigma}$, $\sigma\in\Gamma$,
be (all) the components of the exceptional divisor $D$, and let $\nu_{\sigma}$ be the
divisorial valuations corresponding to them. There was computed the series
$\widehat{P}(\tt; \LLL)$ for the collection of {\bf all} the divisorial valuations $\nu_{\sigma}$.
There were no equations for the series $\widehat{P}(\tt; \LLL)$
for several curve or divisorial plane valuations.
For one plane curve valuation, the series $\widehat{P}(t; \LLL)$ coincides with the
usual, non-generalized, Poincar\'e series (i.\,e., does not depend on $\LLL$).
This is so because, in this case, all the coefficients of the usual Poincar\'e series $P(t)$
are equal to $0$ or $1$.
For one divisorial valuation, this in not the case. Therefore the computation
of the generalized Poincar\'e
series for a divisorial valuation makes sense. However, this case was not considered yet (see Section~\ref{sec:general_divisor} below).

All the considerations in~\cite{Monats} were for valuations on the algebra
$\OO_{\C^2,0}$ (or, equivalently, on the algebra $\C[[x,y]]$ of formal power series).
The situation becomes different (and more complicated) if one considers valuations
on the algebra $\calE_{\K^2,0}$ of germs of holomorphic functions whose Taylor coefficients
are from a (fixed) subfield $\K$ of the field $\C$ of complex numbers (say, from the field
$\R$ of real numbers) or, equivalently, on the algebra $\K[[x,y]]$ of formal power series:~\cite{BLMS},
\cite{MMJ}, \cite{Extensions}.
In this setting, the (classical) Poincar\'e series for one plane curve
or divisorial valuation was computed in~\cite{Extensions}
(for $\K=\R$ in~\cite{BLMS} and~\cite{MMJ}).
In this case, the coefficients in the Poincar\'e series
of one curve valuation may be greater than $1$.
E.\,g., for $\K=\R$ they are equal to $0$, $1$, or $2$.
Here we give equations for the generalized Poincar\'e series
$\widehat{P}(t; \LLL)$ for one plane curve or divisorial valuation obtaining more fine versions of the results of~\cite{Extensions}.

%%%%%%%%%%%%%%%%%%%%%%%%%%%%%%
\section{Generalized Poincar\'e series of valuations over a $\C$-algebra}\label{sect:General} %% 2
%%%%%%%%%%%%%%%%%%%%%%%%%%%%%%
A (discrete, rank one) valuation on a $\C$-algebra $\AA$ is a map
$\nu: \AA\to \Z_{\ge 0}\cup \{+\infty\}$ such that
\begin{enumerate}
\item[1)] $\nu(0)=+\infty$;
\item[2)] $\nu(\lambda f)=\nu(f)$ for $f\in \AA$,
$\lambda\in\C^*:=\C\setminus\{0\}$;
\item[3)] $\nu(f_1+f_2)\ge\min\{\nu(f_1),\nu(f_2)\}$ for $f_1,f_2\in \AA$;
\item[4)] $\nu(f_1 f_2)=\nu(f_1)+\nu(f_2)$ for $f_1,f_2\in \AA$.
\end{enumerate}

Let us recall the definition of the Poincar\'e series of a collection of valuations
on a $\C$-algebra %% from~\cite{CDK}
in terms of the extended semigroup: \cite{Duke}. Let $\nu_1$,
\dots, $\nu_r$ be (discrete rank one) valuations on a
$\C$-algebra $\AA$.
For $\vv=(v_1, \ldots, v_r)\in\Z_{\ge0}^r$, let
$J(\vv)=\{f\in\AA:\nunu(f)\ge\vv\}$,
where $\nunu(f)=(\nu_1(f), \cdots, \nu_r(f))$,
$\nunu(f)\ge\vv$ means that
$\nu_i(f)\ge v_i$ for all $i$.
Let $I_0=\{1, \ldots, r\}$,
$\1=(1,\ldots, 1)\in\Z_{\ge 0}^r$;
for $I\subset I_0$, let
$\1_I\in\Z^r$ be the $r$-tuple
whose $i$\,th component is equal to $1$ if $i\in I$ and is equal to $0$ otherwise.
Let
\begin{equation*}
F_{\vv}:=\left(J(\vv)/ J(\vv+\1)\right)\left\backslash
\bigcup_{I\subset I_0, I\ne \emptyset}
\left(
J(\vv+\1_I)/ J(\vv+\1)
\right)
\right.\,.
\end{equation*}
One can see that
\begin{equation*}
F_{\vv}:=\left(J(\vv)/ J(\vv+\1)\right)\left\backslash
\bigcup_{i=1}^r
\left(
J(\vv+\1_i)/ J(\vv+\1)
\right)
\right.\,,
\end{equation*}
where $\1_i$ is the $r$-tuple with the $i$th component equal
to $1$ and all other compunents equal to $0$.

The {\em extended semigroup}
of the collection of valuations $\{\nu_i\}$ is
$$
\widehat{S}_{\{\nu_i\}}
=\bigcup_{\vv\in\Z_{\ge0}^r}
F_{\vv}\,.
$$
The operation in $\widehat{S}_{\{\nu_i\}}$
is induced by the multiplication in $\AA$.
The spaces $F_{\vv}$ are called {\em fibres of the extended semigroup}.
Each of them is the complement to an arrangement of vector subspaces in a
complex vector space.

The Poincar\'e series $P(\tt)$
of the collection $\{\nu_i\}$
can be defined (or expressed)
by the equation
\begin{equation}\label{eqn:class_Poincare}
P(\tt)=\sum_{\vv\in\Z_{\ge0}^r}\chi(\P F_{\vv})\cdot\tt^{\vv}
\in\Z[[t_1,\ldots, t_r]]\,,
\end{equation}
where $\tt^{\vv}=t_1^{v_1}\cdot\ldots\cdot t_r^{v_r}$,
$\P F_{\vv}:=F_{\vv}/\C^*$
($\C^*=\C\setminus\{0\}$)
is the projectivization of $F_{\vv}$.
The equation~(\ref{eqn:class_Poincare}) makes sense if all the quotients $J(\vv)/J(\vv+\1)$
are finite-dimensional.

Let $\AA=\OO_{\C^2,0}$ be the algebra of germs of holomorphic functions in two variables.
There are essentially two classes of discrete,
rank one valuations on it: curve valuations and divisorial
ones (see, e.\,g., \cite{Spiv}). For a collection of
curve or divisorial valuations on $\OO_{\C^2,0}$,
all the quotients $J(\vv)/J(\vv+\1)$
are finite-dimensional. Thus Equation~(\ref{eqn:class_Poincare}) makes sense.
For $r=1$, that is for one valuation, the fibres $F_{v}$, $v\in\Z_{\ge 0}$,
are punctured (finite dimentional) complex affine spaces. The Euler characteristic of
the projectivization of a complex affine space is equal to its dimension. Thus in this case
Equation~(\ref{eqn:class_Poincare}) reduces to the standard one:
\begin{equation}\label{eqn:Poincare_for_one}
P(t)=\sum_{v=0}^{\infty} \dim_{\C}(J(v)/J(v+1))\cdot t^v\in \Z[[t]].
\end{equation}
For a curve valuation, all the coefficients in (\ref{eqn:Poincare_for_one})
are equal to $0$ or $1$.

Each space $\P F_{\vv}$
is the complement to an arrangement of projective subspaces in a
(finite-dimensional) projective space.
Therefore its generalized Euler characteristic
$\chi_g(\P F_{\vv})$
(the class
$[\P F_{\vv}]\in K_0(\Var)$)
is a polynomial in $\LLL$.
The {\em generalized Poincar\'e series} of the collection $\{\nu_i\}$
is defined by
\begin{equation}\label{eqn:general_Poincare}
  P(\tt;\LLL)=\sum_{\vv\in\Z_{\ge0}^r}\chi_g(\P F_{\vv})\cdot\tt^{\vv}
\in\Z[[t_1,\ldots, t_r;\LLL]]\,.
\end{equation}
(From now on, we shall consider only this version of the generalized Poincar\'e series, and therefore we shall omit the hat over $P$ used
in~\cite{Monats} and in the Introduction.)

\begin{example}
 1) Let us consider two plane curve valuations defined by smooth branches transversal to
 each other. The classical Poincar\'e series $P(t_1,t_2)$ is equal to $1$.
 The semigroup of values in $\Z^2_{\ge 0}$ consists of the origin $(0,0)$ and of
 all the points $(v_1,v_2)$ from $\Z^2_{> 0}$. The fibre $F_{\vv}$ of the extended semigroup is $\C^*=\C\setminus \{0\}$ for $\vv=(0,0)$ and is $(\C^*)^2$ for all
 other $\vv$ in the semigroup of values. The generalized Euler characteristics of their projectivizations are equal to $1$ in the first case and to $\LLL-1$ in the
 second. Therefore one has
 $$
 P(t_1,t_2;\LLL)=1+t_1t_2\frac{\LLL-1}{(1-t_1)(1-t_2)}=
 \frac{1-t_1-t_2+\LLL t_1t_2}{(1-t_1)(1-t_2)}.
 $$

 2) Let us consider the modification of $(\C^2,0)$ by two sequential blow-ups
 and let us consider two divisorial valuations corresponding to the components
 of the exceptional divisor. According to~\cite[Theorem 3]{Monats}, the generalized
 Poincar\'e series is equal to
 $$
 P(t_1,t_2;\LLL)=\frac{1-t_1t_2-t_1t_2^2+\LLL t_1^2t_2^3}{(1-t_1t_2)(1-t_1t_2^2)(1-\LLL t_1t_2)(1-\LLL t_1t_2^2)}.
 $$

 For a little bit more complicated collections of curves or of divisors,
 the equations become much more involved.
 %% Therefore, at the moment, one can see no approach to computation of this generalized Poincar\'e series for an arbitrary collection of valuations.
\end{example}

%%%%%%%%%%%%%%%%%%%%%%%%%%%%%%%%
\section{The generalized Poincar\'e series of valuations on functions over a subfield of $\C$}
\label{sec:over_a_subfield} %% 3
%%%%%%%%%%%%%%%%%%%%%%%%%%%%%%%%
Let $\K$ be a subfield of the field $\C$ of complex numbers and let $\calE_{\K^2,0}$
be the algebra of germs of holomorphic functions on the (complex) plane
whose Taylor coefficients are from $\K$. Valuations on $\calE_{\K^2,0}$
(or, equivalently, on the algebra $\K[x,y]$ of formal power series)
and their Poincar\'e series were considered in~\cite{Extensions}.
A valuation on the algebra $\OO_{\C^2,0}$ of germs of holomorphic functions in two variables
defines a valuation on its subalgebra $\calE_{\K^2,0}$. In fact, each (discrete rank one) valuation
on $\calE_{\K^2,0}$ is the restriction of a valuation on $\OO_{\C^2,0}$: see, e.\,g.,~\cite{Endler}.

If $\AA$ is a $\K$-algebra (e.\,g., $\K[x,y]$) and $\nu_1$, \dots, $\nu_r$ are (discrete rank one) valuations on $\AA$,
the Euler characteristic of $\P F_{\vv}$
$$
(F_{\vv}=\left(J^{\K}(\vv)/J^{\K}(\vv+\1)\right)\setminus
\bigcup_{i=1}^r \left(J^{\K}(\vv+\1_i)/J^{\K}(\vv+\1)\right)\,,
$$
$\J^{\K}(\vv)=\{f\in\AA:\nunu(f)\ge\vv\}$) and thus Equations~(\ref{eqn:class_Poincare}) and~(\ref{eqn:general_Poincare}),
in general, make no sense. However, if the $\K$-vector spaces $J^{\K}(\vv)/J^{\K}(\vv+\1)$ are finite-dimensional,
for the classical Poincar\'e series (an analogue of~(\ref{eqn:class_Poincare})),
one can use the initial definition from~\cite{CDK}. Also in this case
one can define the generalized Euler characteristic
$\chi_g^{\K}(\P F_{\vv})$ as a polynomial in $\LLL$, where $\LLL$ is
the ``class'' of a one-dimensional affine $\K$-space.
This means that one defines the generalized Euler characteristic $\chi_g^{\K}(\K^s)$ of
the affine $s$-dimensional $\K$-space as
$\LLL^s$ and extends the definition to the complements to arrangements of projective subspaces in projective $\K$-spaces using the inclusion-exclusion principle.
In particular, one has
$\chi_g^{\K}(\P\K^s)=\frac{\LLL^{s}-1}{\LLL-1} = 1+\LLL+\ldots+\LLL^{s-1}$.

It is not difficult to see that this notion is well defined.
In other terms, one can define $\chi_g^{\K}(\P F_{\vv})$
by the equation
\begin{equation*}
  \chi_g^{\K}(\P F_{\vv})=\sum_{I\subset I_0}(-1)^{\#I}
  \frac{\LLL^{\dim_{\K}\left(J^{\K}(\vv+\1_I)/J^{\K}(\vv+\1))\right)}-1}{\LLL-1}\,.
\end{equation*}

\begin{definition}
  The {\em generalized Poincar\'e series} of the collection of valuations $\{\nu_i\}$ on $\EE_{\K^2,0}$ is defined by
  \begin{equation}\label{eqn:K_Poincare}
  P^{\K}(\tt;\LLL)=\sum_{\vv\in\Z_{\ge0}^r}\chi_g^{\K}(\P F_{\vv})\cdot\tt^{\vv}
\in\Z[[t_1,\ldots, t_r;\LLL]]\,.
\end{equation}
\end{definition}
For one valuation Equation~(\ref{eqn:K_Poincare}) reduces to
\begin{equation}\label{eqn:K_Poincare_for_one}
 P^{\K}(t;\LLL)=\sum_{v=0}^{\infty}
 \frac{\LLL^{\dim_{\K}(J^{\K}(v)/J^{\K}(v+1))}-1}{\LLL-1}\cdot t^v\,.
\end{equation}

%%%%%%%%%%%%%%%%%%%%%%%%%%%%%%%%
\section{On relations between the usual and the generalized Poincar\'e series}
\label{sec:relations} %% 4
%%%%%%%%%%%%%%%%%%%%%%%%%%%%%%%%
In one direction, the
relation between the usual and the generalized Poincar\'e series is clear. The
generalized Poincar\'e series always determines the usual one, namely,
$P(\tt)=P(\tt,1)=P(\tt;\LLL)_{{\displaystyle{\vert}} \LLL\mapsto 1}$.
In the other direction, a relation is less clear.
In fact, for a ``mixed'' collection consisting both of plane curve and divisorial
valuations, the (classical) Poincar\'e series does not determine the generalized
one. (This follows from the example at the end of~\cite{FAOM}.)
For a collection consisting only of plane curve or only of divisorial valuations
on $\OO_{\C^2,0}$,
the classical Poincar\'e series determines the generalized one.
However, an explicit way to restore $P(\tt;\LLL)$ from
$P(\tt)$ is not known. For example, the coefficient
$0$ in the usual Poincar\'e series may mean $0$, or $\LLL-1$, or
$\LLL^2-1$, {\dots} in the generalized one.

For one valuation, a relation looks clear. Namely, if
\begin{equation}\label{eqn:99}
P(t)=\sum_{v=0}^{\infty}a_vt^v,
\end{equation}
then
%% $$
%% P_g(t;\LLL)=\sum_{v=0}^{\infty}(1+\LLL+\ldots+
%% \LLL^{a_v-1})t^v
%% \left(=\sum_{v=0}^{\infty}\frac{1-\LLL^{a_v}}{1-\LLL}\right)\,.
%% $$
$$
P(t;\LLL)=\sum_{v=0}^{\infty}(1+\LLL+\ldots+
\LLL^{a_v-1})\cdot t^v
\left(=\sum_{v=0}^{\infty}\frac{\LLL^{a_v}-1}{\LLL-1}\cdot t^v\right)\,.
$$
(Pay attention that, for any valuation,
the coefficients $a_v$
are non-negative.)
The problem is that
almost never the (usual) Poincar\'e series $P(t)$ is known as a power series with explicitly defined coefficients.
As a rule, the Poincar\'e series $P(t)$ is computed as a rational function in $t$.
Moreover, in the majority of cases it is expressed as
a finite product/ratio of the binomials $1-t^m$:
\begin{equation}\label{eqn:ACampo_form}
P(t)=\prod_{m\ge 1}(1-t^m)^{s_m}
\end{equation}
with $s_m\in\Z$.
One can describe an algorithm to find the coefficients
in the expression of the Poincar\'e series in the form~(\ref{eqn:99}), 
however, this does not lead to a closed formula for the
coefficients in the generalized Poincar\'e series.
Moreover, having an equation for the Poincar\'e series
of the form~(\ref{eqn:ACampo_form}), it looks natural
to try to get an (or rather the) expression for the
generalized Poincar\'e series in the form
\begin{equation}\label{eqn:gen_ACampo_form}
P(t;\LLL)=\prod_{m\ge 1, \ell\ge 0}(1-\LLL^{\ell}t^m)^{s_{m,\ell}}.
\end{equation}
For a general expression of the form~(\ref{eqn:ACampo_form}) this looks very difficult
(if possible).
Moreover, the algorithm of recovering the generalized Poincar\'e series from the usual one described at the beginning of this section can be applied only to power series with
non-negative coefficients. It is unclear when an
expression of the form~(\ref{eqn:ACampo_form})
represents a power series with non-negative coefficients.

All these make it reasonable not to try to restore
the generalized Poincare\'e series from the (computed) usual ones,
but to try to compute them somewhat independently.
The main aim of this paper is to adapt (or reformulate) the computations
from~\cite{Extensions} to the generalized Poincar\'e
series giving them for one plane curve or divisorial valuation on $\EE_{\K^2,0}$.

%%%%%%%%%%%%%%%%%%%%%%%%%%%%%%%%
\section{Poincar\'e series of a plane valuation on functions over a subfield of $\C$}
\label{sec:general_over_a_field} %% 5
%%%%%%%%%%%%%%%%%%%%%%%%%%%%%%%%
Let $\K$ be a subfield of the field $\C$ of complex numbers, and let $\EE_{\K^2,0}$ be the algebra of germs of holomorphic functions in two variables with the Taylor coefficients from $\K$. (Without changes it is possible to consider,
instead of the algebra $\EE_{\K^2,0}$, the algebra
$\K[[x,y]]$ of formal power series in two variables.)
A (discrete,
rank one) valuation on $\EE_{\K^2,0}$ is the restriction of a
valuation on $\OO_{\C^2,0}$. There are essentially two types of valuations: curve
valuations and divisorial ones.

\medskip

%\subsection*{Curve valuations.}
Let $(C,0)$ be a plane algebroid branch on
$(\C^2,0)$, that is $C$ (after an appropriate change of the coordinates in $\C^2$
defined over $\K$) is given by
\begin{equation}\label{eqn:curve}
x=x(\tau)=\tau^m,\\
y=y(\tau)=\sum_{i\ge m} c_i\tau^i\in \C[[\tau]]\,.
\end{equation}
The parametrization (\ref{eqn:curve}) is assumed to be irreducible, i.\,e., the greatest
common divisor of $m$ and of all
$i$ such that $c_i\ne 0$ is equal to $1$.
The branch $C$ defines a valuation
$\nu=\nu_C$ on the algebra $\OO_{\C^2,0}$ of holomorphic function germs in $x$ and $y$
and thus on the algebra $\EE_{\K^2,0}\subset \OO_{\C^2,0}$.  For
$f\in\OO_{\C^2,0}$, $\nu(f)$ is the degree of the leading term in the series
$f(x(\tau), y(\tau))\in\C[[\tau]]$:
\begin{equation}\label{eqn:curve_valuation}
 f(x(\tau), y(\tau))=a(f)\cdot\tau^{\nu(f)}+
 \text{ terms of higher degree}\;,
\end{equation}
where $a(f)\ne 0$.
If $f(x(\tau), y(\tau))\equiv 0$, $\nu(f):=+\infty$.
Valuations defined this way are called {\em (plane) curve valuations}.

%One can show that, for any plane valuation (a curve or a divisorial one),
%the quotients $J^{\K}(v)/J^{\K}(v+1)$ are finite dimensional (over $\K$)
%and thus Equation~(\ref{eqn:K_Poincare_for_one}) for the generalized
%Poincar\'e series makes sense.

Let us assume that all the coefficients $c_i$ in~(\ref{eqn:curve}) are algebraic over $\K$.
We do not discuss the case when(at least) one of the coefficients
is transcendental over $\K$. In that case the corresponding curve
valuation is in fact a multiple of a divisorial one: see below.

 %% If $\K=\C$, all the coefficients of the series~(\ref{eqn:K-Poincare}) are equal to $0$ or $1$.
 %% This is not the case, in general, for a proper subfield $\K\subset\C$ (see, e.\,g., \cite{BLMS} for $\K=\R$).

\medskip

%\subsection*{The $G$-resolution. }
Let $G$ be the Galois group of the extension $\C/\K$, i.\,e., the group of
automorphisms of $\C$ which are trivial on $\K$.
The group $G$ acts on branches (irreducible complex plane curve germs).
This action can be defined by any of the following three ways.
\begin{enumerate}
 \item[1)] Let a branch $\gamma$ be given by a parametrization
 $x=\tau^k$, $y=\sum\limits_{i\ge k} a_i\tau^i$. Then the branch $g\gamma$, $g\in G$,
is given by the parametrization
 $x=\tau^k$, $y=\sum\limits_{i\ge k} g(a_i)\tau^i$.
 \item[2)] Let a branch $\gamma$ be given by an equation
 $h(x,y)=0$, where $h(x,y)=\sum\limits_{i,j\ge 0} b_{ij}x^iy^j$.
 Then the branch $g\gamma$ is given by the equation $gh(x,y)=0$,
 where $gh(x,y):=\sum\limits_{i,j\ge 0} g(b_{ij})x^iy^j$.
 \item[3)] Let  a branch $\gamma$ be considered as a subset
 of $(\C^2,0)$: $(\gamma,0)\subset (\C^2,0)$. The group $G$
 acts on the plane $\C^2$ by $g(x,y)=(g(x),g(y))$.
 Then the branch $g\gamma$ as a subset of $(\C^2,0)$ is just
 the image of $\gamma$ under this action.
\end{enumerate}

From formal point of view, this description is not absolutely correct (being applied to 
$\EE_{\K^2,0}$); the series written here could be non-convergent. 
We can ignore this problem since, for a bounded values of valuations, we can treat
not functions or series, but jets of high order. 
This means that essentially we can work with polynomials.

\medskip 

A $G$-invariant resolution of the branch $C$ (see~\cite{Extensions}) exists if (and
only if) all the coefficients in~(\ref{eqn:curve})
are from one finite extension of the field $\C$.
Otherwise (if each of them is from a finite extension, but, in general, not all of
them together) there exists a $G$-invariant
resolution process of the branch $C$ (by an infinite sequence of blow-ups).
In the both cases this means the following.

On each step one has a modification of $(\C^2,0)$ endowed with an action of the group $G$.
If the modification is not a $G$-invariant resolution of the curve $C$, i.\,e., if
the total transform of the
curve $GC$ (the union of all the branches from the orbit of $C$ under the
$G$-action) is not a divisor with normal crossings (in particular, this takes place
if the $G$-orbit of $C$ is infinite),
one has to blow up all the intersection points of the strict transform of the curve
$GC$ with the exceptional
divisor simultaneously. These points form the $G$-orbit of a point of the exceptional divisor.
The fact that all the coefficients in~(\ref{eqn:curve}) are algebraic over $\K$
implies that
this orbit is finite. One gets a new modification. This process may either finish
at a moment when one arrives
to a $G$-invariant resolution of the branch $C$ (if all the coefficients are from
one finite extension of $\K$; for short, we shall call this case ``finite'')
or continue without an end (otherwise;
we shall call this case ``infinite'').

The dual graph of the (minimal) $G$-invariant resolution of $C$ (in the finite case) looks
like in Figure~\ref{fig1}.
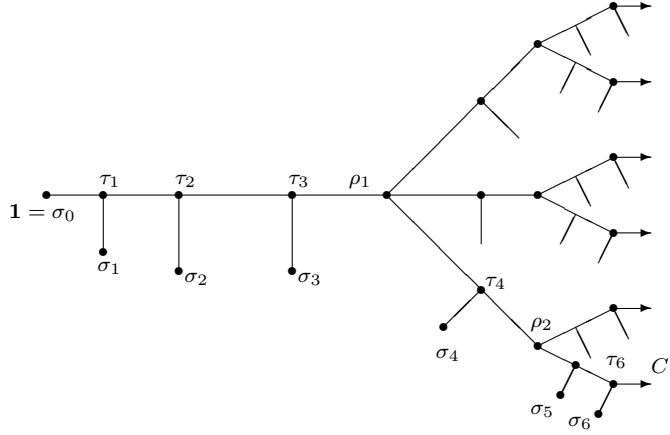
\begin{figure}[h]
$$
\unitlength=0.50mm
\begin{picture}(120.00,110.00)(0,-30)
\thinlines
\put(-30,30){\circle*{2}}
\put(-40,24){{\scriptsize ${\bf 1}=\sigma_0$}}

\put(-30,30){\line(1,0){90}}
\put(60,30){\circle*{2}}

\put(-15,30){\line(0,-1){15}}
\put(-15,30){\circle*{2}}
\put(-15,15){\circle*{2}}
\put(-16.5,10){{\scriptsize$\sigma_1$}}
\put(-16,33){{\scriptsize$\tau_1$}}

\put(5,30){\line(0,-1){20}}
\put(5,30){\circle*{2}}
\put(5,10){\circle*{2}}
\put(6.5,7){{\scriptsize$\sigma_2$}}
\put(4,33){{\scriptsize$\tau_2$}}

\put(35,30){\line(0,-1){20}}
\put(35,30){\circle*{2}}
\put(35,10){\circle*{2}}
\put(36.5,7){{\scriptsize$\sigma_3$}}
\put(34,33){{\scriptsize$\tau_3$}}

\put(50,33){\scriptsize{$\rho_1$}}

\put(60,30){\line(1,1){40}}
\put(85,55){\line(1,-1){10}}
\put(85,55){\circle*{2}}
\put(100,70){\circle*{2}}

\put(100,70){\line(2,1){20}}
\put(120,80){\vector(1,0){10}}
\put(120,80){\circle*{2}}
\put(110,75){\line(1,-2){4}}
\put(120,80){\line(1,-2){4}}

\put(100,70){\line(2,-1){20}}
\put(120,60){\vector(1,0){10}}
\put(120,60){\circle*{2}}
\put(110,65){\line(-1,-2){4}}
\put(120,60){\line(-1,-2){4}}

\put(60,30){\line(1,-1){40}}
\put(100,-10){\circle*{2}}
\put(85,5){\line(-1,-1){10}}
\put(85,5){\circle*{2}}
\put(86,6){{\scriptsize$\tau_4$}}
\put(75,-5){\circle*{2}}
\put(73,-13){{\scriptsize$\sigma_4$}}

\put(98,-5){{\scriptsize$\rho_2$}}

\put(100,-10){\line(2,1){20}}
\put(120,0){\vector(1,0){10}}
\put(120,0){\circle*{2}}
\put(110,-5){\line(1,-2){4}}
\put(120,0){\line(1,-2){4}}

\put(100,-10){\line(2,-1){20}}
\put(120,-20){\vector(1,0){10}}
\put(120,-20){\circle*{2}}
\put(120,-20){\line(-1,-2){4}}
\put(130,-17){{\scriptsize $C$}}
\put(110,-15){\line(-1,-2){4}}
\put(116,-28){\circle*{2}}
\put(118,-15){{\scriptsize$\tau_6$}}
\put(110,-15){\circle*{2}}
\put(106,-23){\circle*{2}}
\put(98,-28){{\scriptsize$\sigma_5$}}
\put(108,-31){{\scriptsize$\sigma_6$}}

\put(60,30){\line(1,0){40}}
\put(85,30){\line(0,-1){13}}
\put(85,30){\circle*{2}}
\put(100,30){\circle*{2}}

\put(100,30){\line(2,1){20}}
\put(120,40){\line(1,-2){4}}
\put(120,40){\vector(1,0){10}}
\put(120,40){\circle*{2}}

\put(110,35){\line(1,-2){4}}
\put(100,30){\line(2,-1){20}}
\put(120,20){\line(-1,-2){4}}
\put(120,20){\vector(1,0){10}}
\put(120,20){\circle*{2}}
\put(110,25){\line(-1,-2){4}}

\end{picture}
$$
\caption{The graph $\Gamma$ of the minimal $G$-resolution.}
\label{fig1}
\end{figure}
If a $G$-resolution does not exist, i.\,e., in the infinite case, in an obvious way
one can define the graph $\Gamma$ of the resolution process. This graph is infinite
as well.

The Galois group $G$ acts on the graph $\Gamma$ and the quotient
$\check \Gamma$ by the action looks like in Figures~\ref{fig2} (in the finite case) or~\ref{fig3}
(in the infinite one).
\begin{figure}[h]
$$
\unitlength=0.50mm
\begin{picture}(120.00,40.00)(0,10)
%%%%%%%%%%%%%%%%%%%%%%%%%

\thinlines
\put(-30,30){\circle*{2}}
\put(-40,24){{\scriptsize ${\bf 1}=\sigma_0$}}

\put(-30,30){\line(1,0){20}}
\put(-8,30){\circle*{0.5}}
\put(-5,30){\circle*{0.5}}
\put(-2,30){\circle*{0.5}}
\put(1,30){\circle*{0.5}}
\put(4,30){\circle*{0.5}}
\put(5,30){\line(1,0){40}}
\put(-15,30){\line(0,-1){15}}
\put(-15,30){\circle*{2}}
\put(-15,15){\circle*{2}}
\put(-16.5,10){{\scriptsize$\sigma_1$}}
\put(-16,33){{\scriptsize$\tau_1$}}

\put(10,30){\line(0,-1){20}}
\put(10,30){\circle*{2}}
\put(10,10){\circle*{2}}
\put(11.5,7){{\scriptsize$\sigma_q$}}
\put(9,33){{\scriptsize$\tau_q$}}

\put(30,30){\circle*{2}}
\put(27,23){\scriptsize{$\rho_1$}}

\put(30,30){\line(1,1){10}}
\put(30,30){\line(1,2){8}}
\put(30,30){\line(0,1){10}}

\put(40,30){\line(0,-1){20}}
\put(40,30){\circle*{2}}
\put(40,10){\circle*{2}}
\put(41.5,7){{\scriptsize$\sigma_{q+1}$}}
\put(39,33){{\scriptsize$\tau_{q+1}$}}

\put(48,30){\circle*{0.5}}
\put(51,30){\circle*{0.5}}
\put(54,30){\circle*{0.5}}

\put(55,30){\line(1,0){40}}

\put(60,30){\line(0,-1){20}}
\put(60,30){\circle*{2}}
\put(60,10){\circle*{2}}

\put(75,30){\line(2,1){10}}
\put(75,30){\line(1,1){10}}
\put(75,30){\line(1,2){8}}
\put(75,30){\line(0,1){10}}

\put(75,30){\circle*{2}}
\put(72,23){\scriptsize{$\rho_2$}}

\put(90,30){\line(0,-1){20}}
\put(90,30){\circle*{2}}
\put(90,10){\circle*{2}}

\put(98,30){\circle*{0.5}}
\put(101,30){\circle*{0.5}}
\put(104,30){\circle*{0.5}}

\put(106,30){\line(1,0){25}}
\put(120,30){\circle*{2}}
\put(118,23){\scriptsize{$\rho_i$}}

\put(110,30){\line(0,-1){15}}
\put(110,30){\circle*{2}}
\put(110,15){\circle*{2}}
\put(111.5,12){\scriptsize{$\sigma_g$}}
\put(108,33){\scriptsize{$\tau_g$}}

\put(120,30){\line(2,1){10}}
\put(120,30){\line(1,1){10}}
\put(120,30){\line(0,1){10}}

\put(134,30){\circle*{0.5}}
\put(137,30){\circle*{0.5}}
\put(140,30){\circle*{0.5}}

\put(142,30){\line(1,0){8}}

\put(150,30){\vector(1,-1){10}}
\put(150,30){\circle*{2}}
\put(162,25){{\scriptsize $C$}}
\put(132,23){\scriptsize{$\rho_s = \delta_C$}}

\put(150,30){\line(2,1){10}}
\put(150,30){\line(1,1){10}}
\put(150,30){\line(1,2){8}}
\put(150,30){\line(0,1){10}}

\end{picture}
$$
\caption{The graph $\check \Gamma$ in the finite case.}
\label{fig2}
\end{figure}
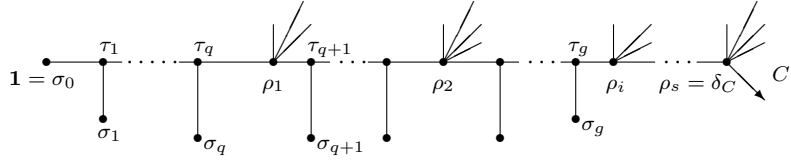
%%%%
%%%%
\begin{figure}[h]
$$
\unitlength=0.50mm
\begin{picture}(120.00,40.00)(0,10)
%%%%%%%%%%%%%%%%%%%%%%%%%

\thinlines
\put(-30,30){\circle*{2}}
\put(-40,24){{\scriptsize ${\bf 1}=\sigma_0$}}

\put(-30,30){\line(1,0){20}}
\put(-8,30){\circle*{0.5}}
\put(-5,30){\circle*{0.5}}
\put(-2,30){\circle*{0.5}}
\put(1,30){\circle*{0.5}}
\put(4,30){\circle*{0.5}}
\put(5,30){\line(1,0){47}}

\put(-15,30){\line(0,-1){15}}
\put(-15,30){\circle*{2}}
\put(-15,15){\circle*{2}}
\put(-16.5,10){{\scriptsize$\sigma_1$}}
\put(-16,33){{\scriptsize$\tau_1$}}

\put(10,30){\line(0,-1){20}}
\put(10,30){\circle*{2}}
\put(10,10){\circle*{2}}
\put(11.5,7){{\scriptsize$\sigma_q$}}
\put(9,33){{\scriptsize$\tau_q$}}

is

\put(30,30){\circle*{2}}
\put(27,23){\scriptsize{$\rho_1$}}

\put(30,30){\line(1,1){10}}
\put(30,30){\line(1,2){8}}
\put(30,30){\line(0,1){10}}

\put(40,30){\line(0,-1){20}}
\put(40,30){\circle*{2}}
\put(40,10){\circle*{2}}
\put(41.5,7){{\scriptsize$\sigma_{q+1}$}}
\put(39,33){{\scriptsize$\tau_{q+1}$}}

\put(48,30){\circle*{0.5}}
\put(51,30){\circle*{0.5}}
\put(54,30){\circle*{0.5}}

\put(55,30){\line(1,0){43}}

\put(60,30){\line(0,-1){20}}
\put(60,30){\circle*{2}}
\put(60,10){\circle*{2}}

\put(75,30){\circle*{2}}
\put(72,23){\scriptsize{$\rho_j$}}

\put(75,30){\line(2,1){10}}
\put(75,30){\line(1,1){10}}
\put(75,30){\line(1,2){8}}
\put(75,30){\line(0,1){10}}

\put(90,30){\line(0,-1){20}}
\put(90,30){\circle*{2}}
\put(90,10){\circle*{2}}
\put(91.5,7){{\scriptsize$\sigma_{g}$}}
\put(89,33){{\scriptsize$\tau_{g}$}}

\put(98,30){\circle*{0.5}}
\put(101,30){\circle*{0.5}}
\put(104,30){\circle*{0.5}}
\put(107,30){\circle*{0.5}}

\put(112,30){\line(1,0){20}}
\put(120,30){\circle*{2}}

\put(120,30){\line(2,1){10}}
\put(120,30){\line(1,1){10}}
\put(120,30){\line(0,1){10}}

\put(134,30){\circle*{0.5}}
\put(137,30){\circle*{0.5}}
\put(140,30){\circle*{0.5}}

\put(142,30){\line(1,0){12}}
\put(150,30){\circle*{2}}
\put(147,23){\scriptsize{$\rho_k$}}

\put(150,30){\line(2,1){10}}
\put(150,30){\line(1,1){10}}
\put(150,30){\line(1,2){8}}
\put(150,30){\line(0,1){10}}

\put(156,30){\circle*{0.5}}
\put(159,30){\circle*{0.5}}
\put(162,30){\circle*{0.5}}
\put(165,30){\circle*{0.5}}

\end{picture}
$$
\caption{The graph $\check \Gamma$ in the infinite case.}
\label{fig3}
\end{figure}
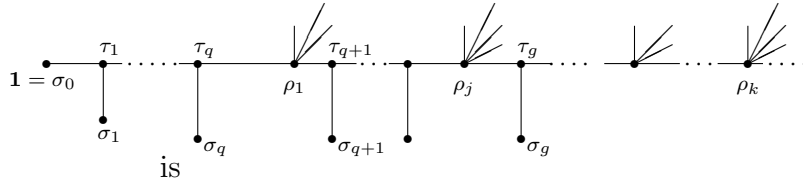
The graph $\Gamma$ contains a subgraph which geometrically coincides with the dual graph of a usual
resolution of the curve $C$  (or of the process of its resolution): the lower part of $\Gamma$ on Figure \ref{fig1}. We use
``short'' notations (by greek letters) for vertices on this subgraph (and also for
the corresponding vertices of the graph $\check \Gamma$). All other vertices of $\Gamma$ are
obtained from them by the action of elements of the Galois group $G$.

\medskip

In Figure~\ref{fig1} (and thus in Figures~\ref{fig2} and~\ref{fig3} as well),
$\sigma_i$, $i=0,1,\ldots, g$, (and also their images under the action of $G$) are
the {\it dead ends\/} of the graph ($g$ is the number of the Puiseux pairs of
the curve $C$),
$\tau_i$,
$i=1,\ldots,g$, (and also their images) are the {\it rupture\/} points of it.

\medskip

Let $G_\sigma$ be the
isotropy subgroup of the component $E_\sigma$, i.\,e., the isotropy subgroup of the
vertex $\sigma$ in $\Gamma$. The subgroup $G_\sigma$ has a finite index in the
Galois group $G$. Let $\K_\sigma\subset \C$ be the invariant subfield of the group
$G_\sigma$. Let
$\rho_j$, $j=1,\ldots,s$, ($s$ is equal to $\infty$ in the infinite case) be the {\it
splitting points\/} corresponding to the changes of the isotropy subgroups
for the $G$-action on $\Gamma$, that means that $G_{\rho_j}\neq G_{\sigma}$ for all
$\sigma > \rho_j$ (note that $G_{\sigma}\subset G_{\rho_j}$).
(Note that they are the splitting points for usual resolution of
different branches from the $G$-orbit of $C$.) In Figures~\ref{fig2} and \ref{fig3}
the small lines at each
$\rho_j$ indicate that the graph after $\rho_j$ is replied at each small
line.

%One can see that, for $g\in G$, $G_{g \sigma}=g G_{\sigma}g^{-1}$,
%$\K_{g \sigma}= g \K_{\sigma}$.

For all the vertices $\sigma\in \Gamma$ in between
the initial one $\sigma_0 =$\;{\bf 1} and $\rho_1$ ($\rho_1$ included) the subgroups
$G_{\sigma}$ and the subfields $\K_\sigma$ are the same:
$G_{\sigma}=G=G_0$, $\K_{\sigma}=\K=\K_0$. For all vertices
$\sigma\in \Gamma$ in between  the vertex $\rho_j$ and $\rho_{j+1}$
($\rho_j$ excluded, $\rho_{j+1}$ included)
the isotropy subgroups $G_{\sigma}$ and the corresponding subfields $\K_{\sigma}$ are the same:
$G_\sigma= G_j$, $\K_\sigma=\K_j$.
In the finite case, for all the vertices $\sigma\in \Gamma$ between $\rho_s$ and $\delta_C$ one has
$G_{\sigma}=G_s$, $\K_\sigma=\K_s$. For vertices from $G$-orbits of the discussed
ones, the corresponding subgroups are conjugate to $G_{\sigma}$ and the subfields
are obtained from $\K_\sigma$ by the shift by an element of the Galois group $G$. One has
$G=G_0\supset G_1\supset \ldots \supset G_s$,
$\K=\K_0\subset \K_1\subset \ldots \subset \K_s$, $[K_j: K_{j-1}]=[G_{j-1}:G_j]$.
Let $\ell_j:= [\K_j: \K_{j-1}]$.

\medskip

%\subsection*{Special curvettes.}
A {\em curvette} corresponding to the
component $E_\sigma$ is the blow-down $C_\sigma$ of a complex-analytic smooth curve
$\gamma_\sigma$ transversal to the
component $E_\sigma$ at a smooth point of the exceptional divisor, that is not an
intersection point with other components. Let $f_\sigma=0$,
$f_\sigma\in \OO_{\C^2,0}$ be an equation of the curve $C_\sigma$.
For a valuation $\nu$ under consideration
and for $\sigma\in \Gamma$, let $m_\sigma = \nu(f_\sigma)$.
One can see that, for the curve valuation $\nu_C$, it is equal to the intersection
number $C\circ C_\sigma$. 
%for
%$\delta_C$ such that, on the surface of the resolution, the
%strict transform of the branch $C$ intersects the component $E_{\delta_C}$.
%In the infinite case, $m_\sigma$ is equal to the intersection number
%$C_{\sigma}\circ C_{\sigma'}$
%for $\sigma'$ on the subgraph of the usual resolution of the curve $C$ far enough from the
%initial component of the modification (in fact, for $\sigma'$ after the vertex $\tau_g$).
(The integers $m_\sigma$ can be interpreted also in terms of the intersection
matrix of the components $E_\sigma$, see, e.\; g., \cite{Extensions}.)

%
%For computations, we shall use coordinate systems compatible with the actions of the
%Galois groups. This means the following.
%There are two standard charts covering the initial component of the resolution
%process: with the coordinates $u=x, w=y/x$ or $u=y, w=x/y$ respectively.
%If, on a previously created component one makes the blow-up at the point
%$(u,w)=(0,a)$ (in a standard chart), one has two standard charts covering the
%new-born component: with the coordinates $u'=u, w'=(w-a)/u$ and
%$u'=w-a, w'=u/(w-a)$ respectively.
%
%On all the components $E_\sigma$ of the exceptional divisor except $E_{\rho_j}$ and
%those from its G-orbit, on each step the blow-up is made at one point
%with the coordinates $(0, a)$ in a standard chart with
%$a\in \K_{\sigma}$.
%In this case it can be convenient to change the coordinate $w$
%by $\widetilde{w}=w/(w-a)$ or
%$\widetilde{w}=(w-a)/w$ so that the coordinates of the intersection points with the
%components on the geodesic between $\sigma_0$ and
%$\delta_C$ become equal to $0$ or to $\infty$ respectively.
%On the component $E_{\rho_j}$,
%the blow-ups are made at a certain point $(0, a)$ with
%$a\in \K_j\setminus \K_{j-1}$ (that is not from
%$\K_{\rho_j}=\K_{j-1}$) and at the  points from its
%$G_{j-1}$-orbit. In this case a change
%as above is inappropriate.

In what follows, we shall use curvettes of special type.
For a component $E_\sigma$,
one can show that the
(smooth and transversal) curve $\gamma_\sigma$ can choosen to be defined over the
field $\K_\sigma$. In other
words, the curve $\gamma_\sigma$ is invariant under the action of $G_\sigma$ (see
\cite{Extensions} for details).

%In particular one can take the curve $\{u=\tau$, $w=c=\text{const}\}$ in a standard
%chart with
%$c\in\K_\sigma$ ($u=0$ is the equation of $E_\sigma$ in the chart).
%{\bf\textcolor{red}{It seems that the paragraph has to be rwritten.}}

\begin{definition}
A $\K_\sigma$-curvette at the component $E_\sigma$ is the blow-down $C_\sigma$ of a
curve $\gamma_\sigma$ described above.
\end{definition}

One can see that $C_\sigma$ has a parametrization like (\ref{eqn:curve}) with all the
coefficients from $\K_\sigma$.
A $\K_\sigma$-curvette at the component $E_\sigma$ can be
defined by an equation $f_\sigma=0$ with $f_\sigma\in \EE_{\K_\sigma^2,0}$.

\begin{definition}
A $G_\sigma$-curvette at the component $E_\sigma$ is the union of the curves
$g C_\sigma$ ($C_\sigma$ is a $\K_\sigma$-curvette at $E_\sigma$) over
representatives of the $G_\sigma$-classes in $G$.
\end{definition}

One can show that a $G_\sigma$-curvette at the component $E_\sigma$ can be defined by
an equation $F_\sigma=0$ with
$F_\sigma\in \EE_{\K^2,0}$. In fact one can take
$F_\sigma=\prod_{[g]\in G/G_\sigma} (gf_\sigma)$, where the product is over
representatives of the $G_\sigma$-classes in $G$.

For $\sigma\in \check{\Gamma}$, let $M_\sigma$ be the value of
$\nu$ on the left-hand side
of the equation $F_\sigma=0$ defining a $G_\sigma$-curvette at the component
$E_\sigma$.
One can see that $M_\sigma=\sum_{[g]\in G/G_{\sigma}} m_{g \sigma}$.

\medskip

%\subsection*{Divisorial valuations.}
Another type of valuations is the divisorial one. Let
$\pi: (W,D)\to (\C^2,0)$ be a $G$-invariant modification of the complex
plane and let $E_{\delta}$ be a component of the surface $W$
of the modification.
%(It is isomorphic to the projective line.)
For a function germ $f\in\OO_{\C^2,0}$, let $\nu_{\delta}(f)$
be the multiplicity of the lifting $\widetilde{f}=f\circ\pi$
of the function $f$ to the surface $W$ of the modification along the
divisor $E_{\delta}$. The map $\nu_{\delta}$ is a valuation on the algebra
$\OO_{\C^2,0}$
(and thus on the algebra $\EE_{\K^2,0}\subset \OO_{\C^2,0}$)
called {\em divisorial}.
The modification $\pi: (W,D)\to (\C^2,0)$ is a resolution of the
valuation $\nu_{\delta}$ on $\EE_{\K^2,0}$. The (dual) graph (or rather its
quotient by the Galois group $G$)
of the minimal resolution of the valuation $\nu_{\delta}$ 
(Figure~\ref{fig4}) looks
like the one of the minimal resolution of a curve valuation
in a finite case (Figure~\ref{fig2}) with the only difference
that the ``last'' vertex is not a splitting one.
In fact it is a resolution of a curvette at the component $E_{\delta}$.

\medskip

Now we can formulate the main results of~\cite{Extensions}.

\begin{theorem}\label{theo:from_Extensions}
For a plane curve valuation $\nu_C$, one has
 \begin{equation}
  P^{\K}_C(t)=\frac{
\prod_{i=1}^g(1-t^{M_{\tau_i}})}{\prod_{i=0}^g(1-t^{M_{\sigma_i}})}\cdot
\prod_{j=1}^s \frac{1-t^{\ell_j M_{\rho_j}}}
{1-t^{M_{\rho_j}}}\,.
 \end{equation}
For a divisorial valuation $\nu_\delta$, one has
 \begin{equation}
  P_{\nu_{\delta}}(t)=
  \frac{1}{1-t^{M_\delta}} \cdot
  \frac{
\prod_{i=1}^g(1-t^{M_{\tau_i}})}{\prod_{i=0}^g(1-t^{M_{\sigma_i}})}\cdot
\prod_{j=1}^s \frac{1-t^{\ell_j M_{\rho_j}}}
{1-t^{M_{\rho_j}}}\,.
 \end{equation}
\end{theorem}

%%%%%%%%%%%%%%%%%%%%%%%%%%%%%%%%
\section{Generalized Poincar\'e series of a plane curve valuation}
\label{sec:general_curve} %% 6
%%%%%%%%%%%%%%%%%%%%%%%%%%%%%%%%
Here we give an equation for the generalized Poincar\'e series
of a curve valuation on $\EE_{\K^2,0}$.

\begin{theorem}\label{theo:for_curve}
For a plane curve valuation $\nu_C$ on $\EE_{\K^2,0}$, one has
 \begin{equation}
  P(t;\LLL)=\frac{
\prod_{i=1}^g(1-t^{M_{\tau_i}})}{\prod_{i=0}^g(1-t^{M_{\sigma_i}})}\cdot
\prod_{j=1}^s \frac{1-\LLL^{\ell_j[K_{j-1}:K_0]}t^{\ell_j M_{\rho_j}}}
{1-\LLL^{[K_{j-1}:K_0]}t^{M_{\rho_j}}}\,.
 \end{equation}
 \end{theorem}

 \begin{proof}
  Essentially we shall modify the proof of Theorem~2
  in~\cite{Extensions} to adapt it to the new setting.
    Let $P(t)=P_{\nu_C}(t)=\sum_{v=0}^{\infty} a_vt^v$,
    $P(t;\LLL)=\sum_{v=0}^{\infty} a_v(\LLL)t^v$ (i.\,e., $a_v=a_v(1)$,
    $a_v(\LLL)=1+\LLL+\LLL^2+\ldots+\LLL^{a_v-1}$).
    For a collection
 $\{k_{\sigma}\}$,
 $\sigma\in\check{\Gamma}$, $k_{\sigma}\in\Z_{\ge0}$, let
 $\calE_{\K^2,0}^{\{k_{\sigma}\}}$ be the set of germs $f\in\calE_{\K^2,0}$
 such that the intersection multiplicity of the strict transform of the zero-level
set
 $\{f=0\}$ with the component $E_{\sigma}$ of the exceptional divisor $D=\pi^{-1}(0)$ is equal
 to $k_{\sigma}$ and, moreover, this strict transform intersects $D$ only at
 smooth points of the total transform $\pi^{-1}(C)$ of the curve $C$.
 Let $\nu(\{k_{\sigma}\}):=\sum_{\sigma\in\check{\Gamma}}k_{\sigma}M_{\sigma}$. One
can see that $\nu(\{k_{\sigma}\})=\nu(f)$
 for any $f\in\calE_{\K^2,0}^{\{k_{\sigma}\}}$.

 Without loss of generality, we assume that, for a fixed $V\in\Z$,
 for all $f\in J^{\K}(v)$ with $v\le V$, the strict transform
 of the curve $\{f=0\}$ intersects the exceptional divisor $D$
 only at smooth points of $\pi^{-1}(C)$.
 This can be achieved by making additional
blow-ups at intersection points of the components of the total transform
 $\pi^{-1}(C)$.

 Let
 $F^{\{k_{\sigma}\}}$ be the image of $\calE_{\K^2,0}^{\{k_{\sigma}\}}$ in
 the quotient $J^{\C}(v)/J^{\C}(v+1)\cong\C$ with
 $v=\nu(\{k_{\sigma}\})$ and let $\overline{F}^{\{k_{\sigma}\}}$
 be the linear span of $F^{\{k_{\sigma}\}}$ over $\K$.

 One has:
 \begin{enumerate}
  \item[1)] ${\displaystyle \bigcup_{\{k_{\sigma}\}: \nu(\{k_{\sigma}\})=v}
  F^{\{k_{\sigma}\}}=
  J^{\K}(v)/J^{\K}(v+1)}$;
  \item[2)] for each collection $\{k_{\sigma}\}$, $F^{\{k_{\sigma}\}}$
  is the complement to an arrangement of vector subspaces 
  in the vector $\K$-space $\overline{F}^{\{k_{\sigma}\}}$.
 \end{enumerate}

Let $d^{\{k_{\sigma}\}}$ be the dimension (over $\K$) of the vector space
$\overline{F}^{\{k_{\sigma}\}}$. One obviously has
$a_v=0$ if there are no $\{k_{\sigma}\}$ with $\nu(\{k_{\sigma}\})=v$. 
If $a_v>0$ then 
$$
a_v=\max_{\{k_{\sigma}\}:\nu(\{k_{\sigma}\})=v}d^{\{k_{\sigma}\}}
{\text{\ \ \ and}}\, \quad
%$$
%$$
a_v(\LLL)=\frac{1-\LLL^{a^v}}{1-\LLL}=1+\LLL+\ldots+\LLL^{a^v-1}\; .
$$

For $j=1,2, \ldots, s+1$ ($s$ may be equal to $+\infty$), let $A_j$ be
the set of collections $\{k_{\sigma}\}$ such that $k_{\sigma}=0$
for all $\sigma$
on the geodesic $[\rho_j,\delta_C]$ from $\rho_j$ to $\delta_C$ (including the ends).
For $j=s+1$ (if $s<+\infty$), we assume this geodesic to be empty.
Let the series
$P^{(j)}(t;\LLL)=\sum_{v=0}^{\infty} a^{(j)}_v(\LLL)\cdot t^v$ be defined by $a^{(j)}_v(\LLL)=0$ if there are no collections
$\{k_{\sigma}\}\in A_j$ with $\nu(\{k_{\sigma}\})=v$ and 
$$
a^{(j)}_v(\LLL)=1+\LLL+\ldots+\LLL^{a^{(j)}_v-1}\text{\ \ with\ \ }
a^{(j)}_v=\max_{\{k_{\sigma}\}\in A_j:
\nu(\{k_{\sigma}\})=v}d^{\{k_{\sigma}\}}\text{\ \ otherwise}\,.
$$
One has $P(t;\LLL)=P^{(s+1)}(t;\LLL)$ if $s<+\infty$ and, if $s=+\infty$
(i.\,e., in the infinite case),
$P(t;\LLL)$ is the limit of $P^{(j)}(t;\LLL)$
for $j\to+\infty$ in the ``$\mathfrak{m}$-adic topology'' ($\mathfrak{m}$
is the maximal ideal in $Z[[t]]$).
This means that, for each $v$, $a_v(\LLL)=a_v^{(j)}(\LLL)$ for $j$
large enough.

Assume first that $\rho_1\ne\sigma_0$.
%%  The necessary changes in the case $\rho_1=\sigma_0$ will be discussed later.

 \begin{lemma}\label{lemma:st-1_for_Theo_curve}
 One has
 \begin{equation}\label{eqn:lemma1}
 P^{(1)}(t;\LLL)=\frac{\prod_{i=1}^g (1-t^{M_{\tau_i}})}
{\prod_{i=0}^g (1-t^{M_{\sigma_i}})}\;.
 \end{equation}
 \end{lemma}

 The statement is a direct consequence of Lemmas~2 and~3
 and Proposition~3
 of~\cite{Extensions} where it is proved that the (usual,
 not generalized) series $P^{(1)}(t)$ is equal to the right
 hand side of Equation~(\ref{eqn:lemma1}) and that all its
 non-zero coefficients are equal to $1$.

 \begin{lemma}\label{lemma:st-2_for_Theo_curve}
  For each $j\ge 1$ and any $v$, $a^{(j)}_v\le [\K_{j-1}:\K_0]$.
  Assume that $v$ is of the form
  $v=\ell M_{\rho_j}+b$ with $b\in S^{\K}_C$ such that
  $b-M_{\rho_j}\notin S_C^{\K}$, that $\ell\le \ell_i-1$,
  and also that $v$ is not of the form $M_{\varkappa}+b'$ with
  $\varkappa>\rho_j$ and $b'\in S^{\K}_C$.
  Then
  $$
  a_{v}^{(j+1)}(\LLL)=
  \frac{\LLL^{\ell[\K_{j-1}:\K_0]}-1}{\LLL-1}
  +\LLL^{\ell[\K_{j-1}:\K_0]}a_{b}^{(j)}(\LLL).
  $$
 \end{lemma}

 This is just the content of Lemmas~4 and~5 of~\cite{Extensions}.

  \begin{lemma}\label{lemma:st-3_for_Theo_curve}
  Let $\varkappa$ be
  on the geodesic between $\rho_j$ and $\delta_C$, $\rho_j$ excluded, i.\,e.\ $\rho_j<\varkappa\le \delta_C$.
  Then
  \begin{enumerate}
    \item[1)] ${\displaystyle a_{M_{\varkappa}}^{(j+1)}(\LLL)=\frac{\LLL^{[\K_{j}:\K_0]}-1}{\LLL-1}}$; 
    \item[2)] $M_{\varkappa}$ is of the form
  $M_{\varkappa}=(\ell_j-1)M_{\rho_j}+b$
  with $b\in S^{\K}_C$ and
  $$a_{b}^{(j)}(\LLL)=\frac{\LLL^{[\K_{j-1}:\K_0]}-1}{\LLL-1}\; .$$
  \end{enumerate}
 % 1)  $a_{M_{\varkappa}}^{(j+1)}(\LLL)=\frac{\LLL^{[\K_{j}:\K_0]}-1}{\LLL-1}$;
 % 2) $M_{\varkappa}$ is of the form
 % $M_{\varkappa}=(\ell_j-1)M_{\rho_j}+b$
 % with $b\in S^{\K}_C$ and
 % $a_{b}^{(j)}(\LLL)=\frac{\LLL^{[\K_{j-1}:\K_0]}-1}{\LLL-1}$.
 \end{lemma}

This is just the content of Lemmas~6 and~7 of~\cite{Extensions}. In particular,
this is so for $\varkappa=\rho_{j+1}$.

  \begin{lemma}\label{lemma:st-4_for_Theo_curve}
  One has
  \begin{eqnarray}
  &\ &P^{(j+1)}(t, \LLL)=\nonumber
  \\
  &=&P^{(j)}(t;\LLL)\cdot
  (1+\LLL^{[K_{j-1}:K_0]}t^{M_{\rho_j}}+
  \LLL^{2[K_{j-1}:K_0]}t^{2M_{\rho_j}}+\ldots+
  \LLL^{(\ell_j-1)[K_{j-1}:K_0]}t^{(\ell_j-1)M_{\rho_j}})
  =\nonumber
  \\
  &=&P^{(j)}(t;\LLL)\cdot\frac{1-\LLL^{\ell_j[K_{j-1}:K_0]}t^{\ell_j M_{\rho_j}}}{1-\LLL^{[K_{j-1}:K_0]}t^{M_{\rho_j}}}.\label{eqn:Prop_for_induction}
  \end{eqnarray}
 \end{lemma}

\begin{proof}
 Equation~(\ref{eqn:Prop_for_induction}) is equivalent to
 \begin{eqnarray}
  a_v^{(j+1)}(\LLL)&=&a_v^{(j)}(\LLL)+
  \LLL^{[K_{j-1}:K_0]}a_{v-M_{\rho_j}}^{(j)}(\LLL)+
  \LLL^{2[K_{j-1}:K_0]}a_{v-2M_{\rho_j}}^{(j)}(\LLL)+\ldots +\nonumber\\
  &{\ }&\ldots + \LLL^{(\ell_j-1)[K_{j-1}:K_0]}a_{v-(\ell_j-1)M_{\rho_j}}^{(j)}(\LLL)\,. \label{eqn:Prop_for_induction_exp}
 \end{eqnarray}
 If $v\in S_C^{\K}$ cannot be represented as $M_{\rho_j}+b$ with
 $b\in S_C^{\K}$, one has $a_v^{(j+1)}(\LLL)=a_v^{(j)}(\LLL)$
 what coincides with Equation~(\ref{eqn:Prop_for_induction_exp}).
 Let $v=\ell M_{\rho_j}+b$ with $b\in S_C^{\K}$, $1\le\ell\le\ell_i-1$,
 and $b-M_{\rho_j}\notin S_C^{\K}$. In this case
 Equation~(\ref{eqn:Prop_for_induction_exp}) tends to
 \begin{eqnarray*}
  a_v^{(j+1)}(\LLL)&=&a_v^{(j)}(\LLL)+
  \LLL^{[K_{j-1}:K_0]}a_{v-M_{\rho_j}}^{(j)}(\LLL)+
  \LLL^{2[K_{j-1}:K_0]}a_{v-2M_{\rho_j}}^{(j)}(\LLL)+
  \ldots \\
  &+& \LLL^{{(\ell-1)}[K_{j-1}:K_0]}a_{v-(\ell-1)M_{\rho_j}}^{(j)}(\LLL)
  + \LLL^{{\ell}[K_{j-1}:K_0]}a_{b}^{(j)}(\LLL)\,.
 \end{eqnarray*}
 (other summands are equal to zero).
 Due to
 Lemma~\ref{lemma:st-3_for_Theo_curve} (Lemma~\ref{lemma:st-1_for_Theo_curve} for $j=1$) one has
 $$
 a_{M_{\rho_j}}^{(j)}(\LLL)=\frac{\LLL^{[\K_{j-1}:\K_0]}-1}{\LLL-1}
 =1+\LLL+\LLL^2+\ldots+\LLL^{[\K_{j-1}:\K_0]-1}\,.
 $$
 For $0\le\ell'\le\ell-1$,
 $v-\ell'M_{\rho_j}=M_{\rho_j}+b^*$ with
 $b^*\in S^{\K}_C$. Therefore all the factors
 $a_{v-(\ell'-1)M_{\rho_j}}^{(j)}(\LLL)$ ($0\le \ell'<\ell$)
 are equal to $a_{M_{\rho_j}}^{(j)}(\LLL)$.
 Thus in this case Equation~(\ref{eqn:Prop_for_induction}) follows
 from Lemma~\ref{lemma:st-2_for_Theo_curve}.

 Now let $v=(\ell_j-1) M_{\rho_j}+b$ with $b\in S^{\K}_C$.
 There are two possibilities.
 \begin{enumerate}
  \item[1)] The value $v$ can be represented as
  $M_{\varkappa}+b'$ with $b'\in S^{\K}_C$,
  $\varkappa>\rho_j$. In this case
  (\ref{eqn:Prop_for_induction}) follows from Lemma~\ref{lemma:st-2_for_Theo_curve}
  and~\ref{lemma:st-3_for_Theo_curve}
  \item[2)] The value $v$ cannot be represented in the indicated form. In this case
 Equation~(\ref{eqn:Prop_for_induction}) follows from
 Lemma~\ref{lemma:st-2_for_Theo_curve} again.
 \end{enumerate}
 \end{proof}

 Lemmas~\ref{lemma:st-1_for_Theo_curve} and~\ref{lemma:st-4_for_Theo_curve} imply the statement
 of Theorem~\ref{theo:for_curve}
 in the case under consideration: $\rho_1\ne\sigma_0$.
 In the case
 $\rho_1=\sigma_0$,
 Lemma~\ref{lemma:st-4_for_Theo_curve}
 holds for $j\ge 2$.
 Lemma~\ref{lemma:st-1_for_Theo_curve}
 does not hold. However, it is possible to compute
 the series $P^{(2)}(t)$ directly
 using Lemmas~\ref{lemma:st-2_for_Theo_curve} and \ref{lemma:st-3_for_Theo_curve}.
 Thus the statement of Theorem~\ref{theo:for_curve}
 holds in this case as well.
 \end{proof}

 %%%%%%%%%%%%%%%%%%%%%%%%%%%%%%%%
\section{Generalized Poincar\'e series of a plane divisorial valuation}
\label{sec:general_divisor} %% 6
%%%%%%%%%%%%%%%%%%%%%%%%%%%%%%%%
Let $\nu_{\delta}$ be the divisorial valuation on $\EE_{\K^2,0}$
defined by a component $E_{\delta}$ of a (finite) $G$-invariant modification of the plane $(\C^2,0)$. The quotient $\check{\Gamma}$ of its dual
graph of the minimal $G$-invariant resolution looks as on Figure~\ref{fig4}.
Pay atention that $\delta\ne\rho_s$.

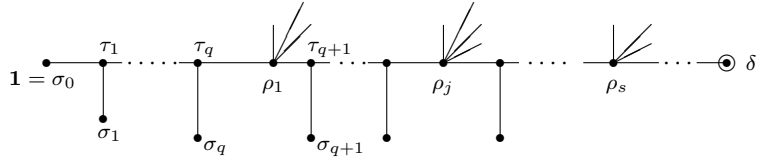
\begin{figure}[h]
$$
\unitlength=0.50mm
\begin{picture}(120.00,40.00)(0,10)
%%%%%%%%%%%%%%%%%%%%%%%%%

\thinlines
\put(-30,30){\circle*{2}}
\put(-40,24){{\scriptsize ${\bf 1}=\sigma_0$}}

\put(-30,30){\line(1,0){20}}
\put(-8,30){\circle*{0.5}}
\put(-5,30){\circle*{0.5}}
\put(-2,30){\circle*{0.5}}
\put(1,30){\circle*{0.5}}
\put(4,30){\circle*{0.5}}
\put(5,30){\line(1,0){40}}

\put(-15,30){\line(0,-1){15}}
\put(-15,30){\circle*{2}}
\put(-15,15){\circle*{2}}
\put(-16.5,10){{\scriptsize$\sigma_1$}}
\put(-16,33){{\scriptsize$\tau_1$}}

\put(10,30){\line(0,-1){20}}
\put(10,30){\circle*{2}}
\put(10,10){\circle*{2}}
\put(11.5,7){{\scriptsize$\sigma_q$}}
\put(9,33){{\scriptsize$\tau_q$}}

\put(30,30){\circle*{2}}
\put(27,23){\scriptsize{$\rho_1$}}

\put(30,30){\line(1,1){10}}
\put(30,30){\line(1,2){8}}
\put(30,30){\line(0,1){10}}

\put(40,30){\line(0,-1){20}}
\put(40,30){\circle*{2}}
\put(40,10){\circle*{2}}
\put(41.5,7){{\scriptsize$\sigma_{q+1}$}}
\put(39,33){{\scriptsize$\tau_{q+1}$}}

\put(48,30){\circle*{0.5}}
\put(51,30){\circle*{0.5}}
\put(54,30){\circle*{0.5}}

\put(55,30){\line(1,0){40}}

\put(60,30){\line(0,-1){20}}
\put(60,30){\circle*{2}}
\put(60,10){\circle*{2}}

\put(75,30){\circle*{2}}
\put(72,23){\scriptsize{$\rho_j$}}

\put(75,30){\line(2,1){10}}
\put(75,30){\line(1,1){10}}
\put(75,30){\line(1,2){8}}
\put(75,30){\line(0,1){10}}

\put(90,30){\line(0,-1){20}}
\put(90,30){\circle*{2}}
\put(90,10){\circle*{2}}
\put(98,30){\circle*{0.5}}
\put(101,30){\circle*{0.5}}
\put(104,30){\circle*{0.5}}
\put(107,30){\circle*{0.5}}

\put(112,30){\line(1,0){20}}
\put(120,30){\circle*{2}}
\put(118,23){\scriptsize{$\rho_s$}}

\put(120,30){\line(2,1){10}}
\put(120,30){\line(1,1){10}}
\put(120,30){\line(0,1){10}}

\put(134,30){\circle*{0.5}}
\put(137,30){\circle*{0.5}}
\put(140,30){\circle*{0.5}}

\put(142,30){\line(1,0){8}}

\put(150,30){\circle{4}}
\put(150,30){\circle*{2}}

\put(155,28){{\scriptsize $\delta$}}

\end{picture}
$$
\caption{The graph $\check \Gamma$ for a divisorial valuation.}
\label{fig4}
\end{figure}

 \begin{theorem}\label{theo:for_divisorial}
For a divisorial valuation $\nu_{\delta}$, one has
 \begin{equation}\label{eqn:theo_divisor}
  P(t;\LLL)=
  \frac{1}{1-\LLL^{[K_s:K_0]} \cdot t^{M_\delta}} \cdot
\frac{\prod_{i=1}^g(1-t^{M_{\tau_i}})}{\prod_{i=0}^g(1-t^{M_{\sigma_i}})}\cdot
\prod_{j=1}^s \frac{1-\LLL^{\ell_j[K_{j-1}:K_0]}t^{\ell_j M_{\rho_j}}}
{1-\LLL^{[K_{j-1}:K_0]}t^{M_{\rho_j}}}\,.
 \end{equation}
\end{theorem}

 \begin{proof}
 Like in Section~\ref{sec:general_curve}, let $F^{\{k_\sigma\}}_{\delta}$ be the image of 
 $\EE^{\{k_\sigma\}}_{\K^2,0}$ in the quotient 
 $J^{\C}(v)/J^{\C}(v+1)$ with $ v = \nu(\{k_\sigma\})$ and let $\overline{F}^{\{k_\sigma\}}_{\delta}$
 be the linear span of $F^{\{k_\sigma\}}_{\delta}$ over $\K$. Pay attention that, in general, 
 $J^{\C}(v)/J^{\C}(v+1)$ is not one-dimensional over $\C$. Let 
 $d^{\{k_\sigma\}}_\delta = \dim_{\K} \overline{F}^{\{k_\sigma\}}_{\delta}$. 
 
Let $C$ be a $\K_s$-curvette at the component $E_{\delta}$.
As usual, for the coefficient $a_v(\LLL)$ in the generalized Poincar\'e series
$P_{\nu}(t;\LLL)=\sum_{v=0}^{\infty} a_v(\LLL) t^v$ one has
$a_v(\LLL)=\frac{\LLL^{a_v}-1}{\LLL-1}$ with
$a_v := \max d_{\delta}^{\{k_\sigma\}}$ where the maximum is taken over the collections
$\{k_\sigma\}$ such that $\nu(\{k_\sigma\}) = \sum k_\sigma M_\sigma=v$.
(If there are no collections $\{k_\sigma\}$ with $\nu(\{k_\sigma\})=v$,
$a_v(\LLL)=a_v=0$.)
Let $A'$ be the set of collections $\{k_\sigma\}$ such
that $k_\delta=0$ and let
$P'(t;\LLL):= \sum a'_{v}(\LLL)t^v$, where
$$
a'_v(\LLL)= \frac{\LLL^{a'_v}-1}{\LLL-1}, \quad
a'_v = \max_{\{k_\sigma\}\in A':\nu(\{k_\sigma\}) = v} d_{\delta}^{\{k_\sigma\}}\,.
$$

\begin{lemma}\label{lemma:1_in_divis}
 One has
 \begin{equation*}
P'(t;\LLL) = P_C(t:\LLL)\,.
\end{equation*}
\end{lemma}

\begin{proof}
For two functions $f_1$ and $f_2$ from $\EE^{\{k_{\sigma}\}}_{\K^2,0}$, $\{k_{\sigma}\}\in A'$,
the ratio $\Psi=\widetilde{f}_1/\widetilde{f}_2$ of their liftings to the surface
of modification is a non-zero constant on the component $E_{\delta}$ of the exceptional divisor. This permits to identify the sets $F^{\{k_{\sigma}\}}_C$ and
$F^{\{k_{\sigma}\}}_{\delta}$. Therefore the statement is a direct consequence of
Remark~9 in~\cite{Extensions}.
\end{proof}

\begin{lemma}\label{lemma:2_in_divis}
  Let $v\in S^{\K}_C$ be of the form
  $v=\ell M_{\delta}+b$ with $b\in S^{\K}_C$ and $b-M_{\delta}\notin S^{\K}_C$.
  Then
  $$
  a_v(\LLL)=\frac{\L^{\ell[\K_s:\K_0]}-1}{\LLL-1}+\LLL^{\ell[\K_s:\K_0]}a_b'(\LLL)\,.
  $$
 \end{lemma}

 This is a direct consequence of Lemma~6 in~\cite{Extensions}.

 \begin{lemma}\label{lemma:3_in_divis}
 One has
 \begin{equation}\label{eqn:in_lemma_3_in_divis}
 P_{\nu}(t;\LLL)={P'}(t; \L)\cdot\frac{1}{1-\LLL^{[\K_s:\K_0]}t^{M_{\delta}}}\,.
 \end{equation}
 \end{lemma}

 \begin{proof}
 Equation~(\ref{eqn:in_lemma_3_in_divis}) is equivalent to
 \begin{equation}\label{eqn:in_lemma_3_in_divis_exp}
  a_v(\LLL)=a_v'(\LLL)+
  \LLL^{[K_{s}:K_0]}a_{v-M_{\delta}}'(\LLL)+
  \LLL^{2[K_{s}:K_0]}a_{v-2M_{\delta}}'(\LLL)+\ldots
 \end{equation}
 Let $v\in S^{\K}$ be of the form $\ell M_{\delta}+b$ with $b\in S^{\K}$,
 $b-M_{\delta}\notin S^{\K}$. In this case Equation~(\ref{eqn:in_lemma_3_in_divis_exp}) reduces to
 \begin{equation*}
  a_v(\LLL)=a_v'(\LLL)+
  \LLL^{[K_{s}:K_0]}a_{v-M_{\delta}}'(\LLL)+
  %\LLL^{2[K_{s}:K_0]}a_{v-2M_{\delta}}'(\LLL)+
  \ldots +
  \LLL^{(\ell-1)[K_{s}:K_0]}a_{v-(\ell-1)M_{\delta}}'(\LLL)+
  \LLL^{\ell[K_{s}:K_0]}a_{b}'(\LLL)\,.
 \end{equation*}
 The latter is a direct consequence of Lemma~\ref{lemma:2_in_divis}
 with the fact that
 $a_{v-\ell'M_{\delta}}(\LLL)=1+\LLL+\LLL^2+\ldots\LLL^{[K_{s}:K_0]-1}$.
 \end{proof}

Lemmas~\ref{lemma:1_in_divis} and~\ref{lemma:3_in_divis} prove the statement of
Theorem~\ref{theo:for_divisorial}
 \end{proof}

 \begin{remark}
 In Introduction it was mentioned that the equation for the generalized Poincar\'e
 series for a divisorial valuation makes sense also for $\K=\C$, i.\,e.,
 for a valuation on $\OO_{\CC^2,0}$. In this case Equation~(\ref{eqn:theo_divisor}) tends to
 $$
 P(t;\L)=\frac{\prod_{i=1}^g(1-t^{m_{\tau_i}})}{\prod_{i=0}^g(1-t^{m_{\sigma_i}})}
 \cdot\frac{1}{1-\LLL t^{m_{\delta}}}\,.
 $$
 As an example, let us take the divisorial valuations whose resolution graphs
 are shown in Figure~{\ref{figlast}}.
 \begin{figure}[h]
$$
\unitlength=1.00mm
\begin{picture}(120.00,25.00)(10,5)
%%%%%%%%%%%%%%%%%%%%%%%%%

\thinlines

\put(10,25){\line(1,0){40}}

\put(10,25){\circle*{2}}
\put(5,20){{${\bf \sigma_0}$}}
\put(30,25){\circle*{2}}
\put(30,25){\line(0,-1){20}}
\put(30,5){\circle*{2}}
\put(25,20){{${\bf \tau_1}$}}
\put(24,5){{${\bf \sigma_1}$}}

\put(50,25){\circle{3}}
\put(50,25){\circle*{1}}

\put(50,25){\line(0,-1){20}}
\put(50,5){\circle*{2}}
%\put(50,5){\circle{3}}
\put(44,28){{${\bf \tau_2=\delta}$}}
\put(44,5){{${\bf \sigma_2}$}}

%SECOND

\put(70,25){\line(1,0){60}}

\put(70,25){\circle*{2}}
\put(65,20){{${\bf \sigma_0}$}}
\put(90,25){\circle*{2}}
\put(90,25){\line(0,-1){20}}
\put(90,5){\circle*{2}}
\put(85,20){{${\bf \tau_1}$}}
\put(84,5){{${\bf \sigma_1}$}}

\put(110,25){\circle*{2}}

\put(110,25){\line(0,-1){20}}
\put(110,5){\circle*{2}}
\put(104,20){{${\bf \tau_2}$}}
\put(104,5){{${\bf \sigma_2}$}}

%\put(110,25){\circle{3}}
%\put(110,25){\circle*{1}}

\put(130,25){\circle{3}}
\put(130,25){\circle*{1}}
%
%\put(69,20){{${\bf \delta}$}}

\put(129,18){{${\bf \delta}$}}

\end{picture}
$$
\caption{Resolution graphs of the divisorial valuations under consideration}
\label{figlast}
\end{figure}
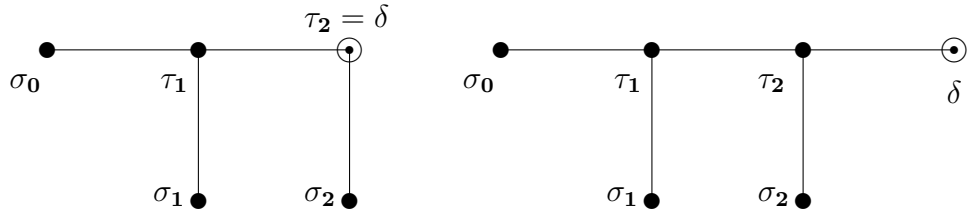
 For each of them the corresponding curvetta
 is of the form $x=t^4$, $y=t^6+t^7$. For the first one, one has
 $$
 P(t)=\frac{(1-t^{12})}{(1-t^4)(1-t^6)(1-t^{13})}\,.\quad
 P(t;\LLL)=\frac{(1-t^{12})(1-t^{26})}{(1-t^4)(1-t^6)(1-t^{13})(1-\LLL t^{26})}\,.
 $$
 For the second one,
 $$
 P(t)=\frac{(1-t^{12})(1-t^{26})}{(1-t^4)(1-t^6)(1-t^{13})(1-t^{27})}\,.\quad
 $$
 $$
 P(t;\LLL)=\frac{(1-t^{12})(1-t^{26})}{(1-t^4)(1-t^6)(1-t^{13})(1-\LLL t^{27})}\,.
 $$
 
 This example supports the opinion that it is difficult to formulate a general rule to get the series $P(t;\L)$ of the form~(\ref{eqn:gen_ACampo_form})  from a representation of the series 
 $P(t)$ in the form~(\ref{eqn:ACampo_form}). 
 
 \end{remark}

Addresses:

F. Delgado:
IMUVA (Instituto de Investigaci\'on en
Matem\'aticas), Universidad de Valladolid,
Paseo de Bel\'en, 7, 47011 Valladolid, Spain.
\newline E-mail: fdelgado\symbol{'100}uva.es

S.M. Gusein-Zade:
Moscow State University, Faculty of Mathematics and Mechanics, Leninskie Gory 1, Moscow, GSP-1, 119991, Russia.\\
\& National Research University ``Higher School of Economics'',
Usacheva street 6, Moscow, 119048, Russia.
\newline E-mail: sabir\symbol{'100}mccme.ru

\end{document}